\documentclass[a4paper]{amsart}
\usepackage{hyperref}   
\usepackage{amssymb,xcolor}
\usepackage{stmaryrd} % \owedge
\usepackage{enumerate}
\newcommand{\unfoldedcomment}[7]{}
\newcommand{\foldedcomment}[7]{}

\usepackage{mathtools}  
\mathtoolsset{showonlyrefs}
\newtheorem{theorem}{Theorem}[section]
\newtheorem{corollary}[theorem]{Corollary}
\newtheorem{definition}[theorem]{Definition}
\newtheorem{lemma}[theorem]{Lemma}
\newtheorem{remark}[theorem]{Remark}
\newtheorem{proposition}[theorem]{Proposition}
\numberwithin{equation}{section}
\numberwithin{figure}{section}
\begin{document}

\title{Llarull type theorems for bands in Three and Four dimensions}

\author{Xiaoxiang Chai}
\address{School of Mathematics and Statistics, Central China Normal
University, Wuhan, China}
\address{Department of Mathematics, POSTECH, Pohang, Gyeongbuk 37673, South
Korea}
\email{xxchai@kias.re.kr}

\author{Xueyuan Wan}
\address{Mathematical Science Research Center, Chongqing University of
Technology, Chongqing 400054, China}
\email{xwan@cqut.edu.cn}

\begin{abstract}
  Llarull's theorem asserts that the scalar curvature and the metric on the
  $n$-sphere cannot be bounded below at the same time by those of the standard
  $n$-sphere. Using the warped $\mu$-bubble method, we develop Llarull type
  theorems for three and four-dimensional bands with spectral scalar curvature
  bounds.
\end{abstract}



{\maketitle}

\section{Introduction}

The standard $n$-sphere has the Llarull type rigidity.

\begin{theorem}[{\cite{llarull-sharp-1998}}]
  \label{std llarull}Let $f : (M, g) \to (\mathbb{S}^n, g_{\mathbb{S}^n})$ be
  a distance non-increasing spin map such that $R_g \geqslant n (n - 1)$, then
  $f$ is an isometry.
\end{theorem}

Here, $R_g$ denotes the scalar curvature. In recent years, the theorem was
established via different techniques: spinors {\cite{bar-scalar-2024}},
{\cite{li-spectral-2024}}, {\cite{wang-scalar-2025}}, spacetime harmonic
function {\cite{hirsch-rigid-2025}}, $\mu$-bubble {\cite{gromov-four-2023}},
{\cite{hu-rigidity-2023}}, {\cite{cecchini-scalar-2024}}. See also
{\cite{lott-index-2021}}, {\cite{goette-scalar-2002}} and
{\cite{listing-scalar-arxiv-2010}} for generalizations. We shall revisit
{\cite{cecchini-scalar-2024}} and {\cite{listing-scalar-arxiv-2010}} later.

Among many of the works listed here, a Llarull type theorem for warped
products with a log-concave warping factor or \text{{\itshape{bands}}} was
also established. We now introduce the manifold with a band structure. A
\text{{\itshape{band}}} is a manifold $M$ with at least two boundary
components which are re-grouped into two non-empty groups $\partial_- M$ and
$\partial_+ M$. We say that a band is \text{{\itshape{over-spherical}}} if
there is a map $f : M \to \mathbb{S}^n_I : = I \times \mathbb{S}^{n - 1}$ such
that $f (\partial_{\pm} M) = \{t_{\pm} \} \times \mathbb{S}^{n - 1}$. Here,
$I$ is some closed interval $[t_-, t_+]$. We also say that $f$ is
over-spherical. For a hypersurface $\Sigma$ homologous to $\partial_- M$, we
fix the direction of the unit normal of $\Sigma$ to point outside of the
region bounded by $\Sigma$ and $\partial_- M$. In particular, In particular,
we fix $\nu_-$ (\text{{\itshape{resp.}}} $\nu_+$) to be the unit normal of
$\partial_- M$ (\text{{\itshape{resp.}}} $\partial_+ M$) to point inside
(\text{{\itshape{resp.}}} outside) of $M$. And the second fundamental form $A$
and the mean curvature $H$ of a given hypersurface $\Sigma$ are computed using
the choice of the unit normals, that is, $A = \nabla \nu$ and $H
=\ensuremath{\operatorname{tr}}_{\Sigma} (\nabla \nu)$.

Let
\begin{equation}
  g_{\xi} = \mathrm{d} t^2 + \xi (t)^2 g_{\mathbb{S}^{n - 1}} \label{warped
  product}
\end{equation}
on $\mathbb{S}^n_I$ and $\xi : I \to \mathbb{R}$ is a positive function.
Define
\begin{equation}
  h_{\xi} (t) = (n - 1) \tfrac{\xi' (t)}{\xi (t)} \label{h} .
\end{equation}
The log-concavity says that $(\log \xi)'' < 0$.

The Llarull type theorem for the warped products is given in the following:

\begin{theorem}
  \label{llarull band}Let $f$ be an over-spherical spin map from $(M, g)$ to
  $(\mathbb{S}^n_I, g_{\xi})$ such that $R_g \geqslant f^{\ast} R_{g_{\xi}}$
  in $M$, $H_{\partial_+ M} \geqslant h_{\xi} (t_+)$ along $\partial_+ M$ and
  $H_{\partial_- M} \leqslant h_{\xi} (t_-)$ along $\partial_- M$, then $f$ is
  an isometry.
\end{theorem}

The $\mu$-bubble proof {\cite{gromov-four-2023}}, {\cite{hu-rigidity-2023}} of
Theorem \ref{llarull band} was previously known to work in three dimensions
with the recent development of Cecchini-Wang-Xie-Zhu
{\cite{cecchini-scalar-2024}} proof of Theorem \ref{std llarull} in dimension
four using Listing's theorem {\cite{listing-scalar-arxiv-2010}} and the Ricci
flow. However, Theorem \ref{llarull band} remains open for higher dimensions
for non-spin bands.

In this article, we settle the four-dimensional case of Theorem \ref{llarull
band} in the non-spin setting, see Theorem \ref{spec llarull band}. In fact,
we can prove a more general theorem which involves the so-called
\text{{\itshape{spectral}}} \text{{\itshape{scalar}}}
\text{{\itshape{curvature}}}.

\begin{definition}
  \label{def spec}Let $\gamma$ be a real number and $(M, g)$ be a Riemannian
  manifold and $u$ be a positive function, we call
  \begin{equation}
    \Lambda (g, u) := - \gamma u^{- 1} \Delta_g u + \tfrac{1}{2} R_g
    \label{spec sc}
  \end{equation}
  the spectral scalar curvature. We also use the convention $\Lambda_{g, u} =
  \Lambda (g, u)$. We omit the reference to $u$ when the context is clear.
\end{definition}

The earliest occurrence of the spectral scalar curvature seems to be in the
work of Schoen-Yau {\cite{schoen-existence-1983}}. More generally, let
$\sigma$ be a real number, we can consider
\begin{equation}
  - \gamma u^{- 1} \Delta_g u + \tfrac{1}{2} R_g + \sigma \gamma u^{- 2} |
  \nabla u|^2 \label{lambda}
\end{equation}
where a special case of \eqref{lambda} is the Perelman or weighted scalar
curvature $P = R_g + 2 \Delta f - | \nabla f|^2$ (with $\sigma = 1 - \gamma /
2$, $f = - \gamma \ln u$), see {\cite{wang-rigidity-2025}} and the references
therein. In fact, \eqref{lambda} is a special case of \eqref{spec sc} because
of the following identity
\[ - u^{- 1} \Delta_g u + \sigma u^{- 2} | \nabla u|^2 = - \tfrac{1}{1 -
   \sigma} u^{- (1 - \sigma)} \Delta_g u^{1 - \sigma} . \]

A spectral Llarull theorem was first established in
{\cite{chai-spectral-2024}} with \eqref{spec sc} bounded below by a positive
constant via spinors and spacetime harmonic functions. There were also similar
results which deals with bounds on \eqref{spec sc}, see
{\cite{chai-band-2025}}, {\cite{hao-llarull-2024}},
{\cite{hirsch-spectral-2024}}, {\cite{zhou-weighted-2025}},
{\cite{chow-scalar-2025}}. Motivated by the warped $\mu$-bubble proof of the
band width estimates {\cite{chai-band-2025}}, in particular for over-torical
bands, we are interested in a non-spin analog of Theorem \ref{llarull band}
for over-spherical bands with spectral scalar curvature
bounds.{\foldedcomment{+1KB5L6S61rv0zAOA}{+1KB5L6S61rv0zAOB}{comment}{bk21}{1766830370}{}{
Theorem \ref{llarull band} was proven by the $\mu$-bubble method in
{\cite{hu-rigidity-2023}} for dimension three. While the band width estimate
for over-torical bands holds up to dimension seven (due to the singularity of
area-minimising hypersurfaces), the $\mu$-bubble proof of Theorem \ref{llarull
band} seems only working well in three dimensions with the exception of
Theorem \ref{std llarull} which was proven by Cecchini-Wang-Xie-Zhu
{\cite{cecchini-scalar-2024}} in dimension four by a combination of Listing's
theorem and the Ricci flow.}}

We are able to show the following Llarull type theorem for bands in three and
four dimensions. Define
\begin{align}
u_{\xi} (t) & = \xi^{\tfrac{1}{2 - \gamma}}, \\
m_{\xi} (t) & = h_{\xi} (t) + \gamma \tfrac{u_{\xi}' (t)}{u_{\xi} (t)} = (n
- 1 + \tfrac{\gamma}{2 - \gamma}) \tfrac{\xi'}{\xi} \label{m}
\end{align}
for \eqref{warped product}. We require that $m_{\xi}' (t) < 0$ on $[t_-,
t_+]$; $\gamma < 4$ when $n = 3$ and $\gamma < 3$ when $n = 4$.

\begin{theorem}
  \label{spec llarull band}Let $n = 3$ or 4, $(M, g)$ be an $n$-dimensional
  oriented band, $f : (M, g) \to (\mathbb{S}^n_I, g_{\xi})$ be a smooth map of
  non-zero degree which does not increase the distance and $f (\partial_{\pm}
  M) \subset f (\partial_{\pm} \mathbb{S}^n_I)$. In addition, there exists a
  positive function $u$ such that the following curvature bounds
\begin{enumerate}[(a)]
    \item $\label{spectral sc}$ $\Lambda_{g, u} \geqslant f^{\ast} \Lambda
    (\xi, u_{\xi})$ in $(M, g)$,
    
    \item $H_{\partial_+ M} + \gamma u^{- 1} \langle \nabla u, \nu_+ \rangle
    \geqslant m_{\xi} (t_+)$ on $\partial_+ M$ and $H_{\partial_- M} + \gamma
    u^{- 1} \langle \nabla u, \nu_- \rangle \leqslant m_{\xi} (t_-)$ on
    $\partial_- M$,
  \end{enumerate}
  hold, then $f$ is an isometry and $u$ is a constant multiple of $u_{\xi}$.
\end{theorem}

\begin{remark}
  When $\gamma = 2$, the definition of $u_{\xi}$ might have a minor issue
  which can be easily addressed by replacing $u_{\xi}$ by $u (t)$, $\xi$ by a
  constant and $m_{\xi}$ by $(2 (n - 1) - (n - 2) \gamma) u' / u$ in Theorem
  \ref{spec llarull band}.
\end{remark}

\begin{remark}
  Note that also the warping factor in \eqref{warped product} is not
  necessarily log-concave, instead the condition is $m_{\xi}' < 0$, and $\xi$
  can be log-convex if $\gamma > 2$.
\end{remark}

Our proof is based on the warped $\mu$-bubble method, which is a
generalization of the $\mu$-bubble method and it recently was used in the
resolution of several important problems together with the spectral version of
various curvatures, see
{\cite{chodosh-generalized-2024,chodosh-stable-2023,chodosh-stable-2024}}. We
use the Gauss-Bonnet theorem when $n = 3$ and Listing's Theorem \ref{thm
listing} when $n = 4$. Listing's theorem which we now recall is stronger than
Llarull's Theorem \ref{std llarull}.

\begin{theorem}[{\cite{listing-scalar-arxiv-2010}}]
  \label{thm listing}If $f : (M^n, g) \to (\mathbb{S}^n, g_{\mathbb{S}^n})$ is
  a smooth spin map of non-zero degree such that
  \[ R_g (p) \geqslant n (n - 1) \| \mathrm{d} f_p \|^2 \text{ for all } p \in
     M, \]
  then there exists a constant $c > 0$ such that $f : (M, c g) \to
  (\mathbb{S}^n, g_{\mathbb{S}^n})$ is an isometry. In particular, when $n =
  3$ (since all 3-manifolds are spin).
\end{theorem}

Our use of Listing's theorem in four-dimensional Llarull type theorem is
different from {\cite{cecchini-scalar-2024}}, more specifically, we are able
to avoid the Ricci flow by exploiting the full rigidity statement of Listing's
Theorem \ref{thm listing}, see Propositions \ref{local rigidity} and \ref{sign
of leaf}.
{\foldedcomment{+YwZVIEztzfSaNy}{+YwZVIEztzfSaNz}{comment}{bk21}{1772019191}{}{\begin{remark}
  It should be quite straightforward to extend our method to some of the
  boundary cases, in particular the rotationally symmetric case with barriers,
  see {\cite{chai-scalar-2023}} (cf. {\cite{ko-scalar-2024}},
  {\cite{wang-scalar-mean-2024}}), see also a recent more general version by
  Ko-Yao {\cite{ko-capillary-2026}} who developed the capillary minimal
  slicing technique. 
\end{remark}}}We also showed a simple generalization of Theorem \ref{thm
listing} in Theorem \ref{spectral listing theorem} which recovers a special
case obtained for the Perelman scalar curvature by Zhou-Zhu
{\cite{zhou-weighted-2025}} via spinor methods.

As Llarull's Theorem \ref{std llarull} is the most important case, we also
need to develop a strategy. To this end, we need a new setup for the function
$\xi$ appearing in the warped product \eqref{warped product}. Now we assume
that $t_- = 0$, \ $\xi$ is positive on $(t_-, t_+]$ and $\xi (t) = A (t - t_-)
+ O (|t - t_- |^2)$ near $t_-$ for some positive constant $A$. Topologically,
$\{t_- \} \times \mathbb{S}^{n - 1}$ is a conical point $p_{\xi}$ in
$\mathbb{S}^n_I$. Now we setup $M$ such that it has similar structures. We
assume that $M$ is a compact metric space such that $M\backslash (\cup_i \{p_i
\})$ is a smooth Riemannian manifold with the Riemannian metric $g$, and $M$
is the metric completion of $(M\backslash \cup_i \{p_i \}, g)$. Let the
resulting metric be $d_g$ and we call $p_i$ a singular point of $M$. We assume
that near $p_i$, the metric $g$ is of the form $\mathrm{d} t^2 + t^2 g_S +
g_1$ for some closed Riemannian manifold $S$ with the Riemannian metric $g$
for some small (compared to $\mathrm{d} t^2 + t^2 g_S$) 2-form $g_1$, and that
$\{t = 0\}$ corresponds to the singular point $p_i$. The limit of the spaces
$\{(M, \lambda^{- 1} d_g, p_i)\}_{\lambda > 0}$ as $\lambda \to 0$ in the
pointed Gromov-Hausdorff distance is called the tangent cone of $M$ at $p_i$.
It is clear that the tangent cone metric is given by $\mathrm{d} t^2 + t^2
g_S$. For the notion of tangent cones, see {\cite[Chapter
8]{burago-course-2001}}.

For simplicity, we assume that $f^{- 1} (p_{\xi})$ exhausts all the conical
points of $(M, g)$. Note that $f^{- 1} (p_{\xi})$ can include smooth points of
$M$, but such points still have the structure described above with $(S, g_S)$
being the standard $(n - 1)$-sphere. Let $p \in f^{- 1} (p_{\xi})$, and we
further assume that the tangent cones of $(M, g)$ at $p$ and of
$(\mathbb{S}^n_I, g_{\xi})$ at $p_{\xi}$ are isometric. Let $\bar{u} (t) =
t^{\tfrac{1}{2 - \gamma}}$, now we assume that $\gamma < 2$, and we require
that the tangent cone metric given by $\bar{g} = \mathrm{d} t^2 + A^2 t^2
g_{\mathbb{S}^{n - 1}}$ satisfies
\begin{equation}
  (n - 2) (A^{- 2} - 1) + \tfrac{2 (n - 1)}{2 - \gamma} \geqslant 0.
  \label{unit cross-section size}
\end{equation}
\begin{theorem}
  \label{llarull cone theorem}Let $n = 3$ or $4$, and let $f$, $(M, g)$ and
  $(\mathbb{S}^n_I, g_{\xi})$ be given as above, if there exists a positive
  function $u$ on $M\backslash f^{- 1} (p_{\xi})$ which satisfies
\begin{enumerate}[(a)]
    \item $\Lambda_{g, u} \geqslant f^{\ast} \Lambda (g_{\xi}, u_{\xi})$ in
    $(M, g)$,
    
    \item $H_{\partial_+ M} + \gamma u^{- 1} \langle \nabla u, \nu_+ \rangle
    \geqslant m_{\xi} (t_+)$ on $\partial_+ M$,
    
    \item $u = t^{\alpha} v (x) + O (t^{1 + \alpha})$ for some $\alpha \in
    \mathbb{R}$ near $t = 0$,
  \end{enumerate}
  then $f$ is an isometry and $u$ is a constant multiple of $u_{\xi}$.
\end{theorem}

\begin{remark}
  The 2-tensor $\mathcal{R}_{\bar{g}}$ defined in \eqref{mod ric} of the
  metric $\bar{g}$ is non-negative under the assumption of \eqref{unit
  cross-section size}; note that $\bar{g} = \mathrm{d} t^2 + A^2 t^2
  g_{\mathbb{S}^{n - 1}}$ might not have a positive curvature operator since
  $A$ might be greater than 1. On the other hand, \eqref{unit cross-section
  size} is always satisfied for $0 < A \leqslant 1$.
\end{remark}

We believe the assumption of isometric tangent cones are removable, but at the
moment it seems challenging to prove Theorem \ref{llarull cone theorem}
without the isometric tangent cone assumption for general $\gamma \neq 0$
because of $u$. Actually, similar conditions were imposed in
{\cite{chai-spectral-2024}}. For $\gamma = 0$, however, we can remove the
assumption by making use of the distance comparison.

\begin{theorem}
  \label{cone llarull}Let $n = 3$ or $4$, and let $f$, $(M, g)$ and
  $(\mathbb{S}^n_I, g_{\xi})$ be given as above. Assume that $\xi = A t + O
  (t^2)$ with $0 < A \leqslant 1$,
\begin{enumerate}[(a)]
    \item  $R_g \geqslant f^{\ast} R_{g_{\xi}}$ in $(M, g)$,
    
    \item and $H_{\partial_+ M} \geqslant h_{\xi} (t_+)$ on $\partial_+ M$,
  \end{enumerate}
  then $f$ is an isometry.
\end{theorem}

Our construction is local, and it also applies to the case when $\xi = A_+
(t_+ - t) + O (|t - t_+ |^2)$ for $0 < A_+ \leqslant 1$ near $t = t_+$. So
this reproves Llarull's Theorem \ref{std llarull} in four dimensions by taking
$\xi : [0, \pi] \to \mathbb{R}$ to be $\sin t$. The three dimensional version
of Theorem \ref{cone llarull} already appeared in {\cite[Theorem
1.6]{chai-scalar-2025}}.

\

The article is organized as follows:

In Section \ref{sec basics}, we introduce the warped $\mu$-bubble, compute the
related first and second variation formulas, in particular, we relate the
second variation with the curvature \eqref{lambda}. We also calculate some
model metrics.

In Section \ref{sec band}, we develop a foliation analysis to prove Theorem
\ref{spec llarull band}. In Section \ref{sec cone}, we construct a foliation
near the conical point with its leaf satisfying a certain mean curvature
condition to show Theorem \ref{llarull cone theorem}. In particular, we prove
Theorem \ref{cone llarull}.

\

\text{{\bfseries{Acknowledgement}}} X.C. has been partially supported by the
National Research Foundation of Korea (NRF) grant funded by the Korea
government (MSIT) (No. RS-2024-00337418). X.W. is sponsored by the National
Key R\&D Program of China (Grant No. 2024YFA1013200) and the Natural Science
Foundation of Chongqing, China (Grant No. CSTB2024NSCQ-LZX0040,
CSTB2023NSCQ-LZX0042).

\section{Basics of warped $\mu$-bubbles}\label{sec basics}

In this section, we introduce the warped $\mu$-bubble which is a critical
point of an energy functional related to the weighted area. We compute the
first and second variation of the energy functional, in particular, we relate
the second variation with the curvature \eqref{lambda}.

\subsection{The variational problem}

We assume that $M$ lies in a slightly larger manifold and we move $\partial_-
M$ slightly outward and obtain a resulting band $M_1$. Let $\mathcal{C}_1$ be
the collection of all Caccioppoli sets $\Omega$ such that $\Omega$ contains a
neighborhood of $\partial_- M_1$ and define the warped $\mu$-bubble functional
\begin{equation}
  E (\Omega) = \int_{\partial^{\ast} \Omega \cap
  \ensuremath{\operatorname{int}}M_1} u^{\gamma} \mathrm{d} \mathcal{H}^{n -
  1} - \int_{\Omega} u^{\gamma} \mu \mathrm{d} \mathcal{H}^n, \label{eq E
  band}
\end{equation}
where $\partial^{\ast} \Omega$ denotes the reduced boundary of $\Omega$. Let
$\mathcal{C}$ be the sub-collection of Caccioppoli sets $\Omega$ such that
$\partial^{\ast} \Omega \backslash \partial_- M$ is a subset in the closure of
$M$. We consider the minimisers of $E$ in the class $\mathcal{C}$. In Gromov's
terminology {\cite[Section 1.6.4]{gromov-four-2023}}, a minimiser of $E$ is
called a \text{{\itshape{warped $\mu$-bubble}}}.

Let $\Omega_t$ be a smooth family of Caccioppoli sets, then the first
variation of $E (\Omega_t)$ is given by
\begin{equation}
  \tfrac{\mathrm{d}}{\mathrm{d} t} E (\Omega_t) |_{t = 0} = \int_{\Sigma} (H +
  \gamma u^{- 1} u_{\nu} - \mu) u^{\gamma} \phi \mathrm{d} \mathcal{H}^{n -
  1}, \label{eq first variation}
\end{equation}
where $\Sigma_t = \partial \Omega_t \cap \ensuremath{\operatorname{int}}M$ and
$H$ is the mean curvature of $\Sigma_t$. The quantity $H + \gamma u^{- 1}
u_{\nu}$ is called the \text{{\itshape{$u^{\gamma}$-weighted mean
curvature}}}, and we often omit the dependence on $u^{\gamma}$.

Let $\Omega = \Omega_0$, if $\Sigma := \partial \Omega \cap
\ensuremath{\operatorname{int}}M$ satisfies $H + \gamma u^{- 1} u_{\nu} - \mu
= 0$ along $\Sigma$. We call $\Sigma$ a $\mu$-hypersurface. Then the second
variation of $E (\Omega_t)$ is given by
\begin{equation}
  \tfrac{\mathrm{d}^2}{\mathrm{d} t^2} E (\Omega_t) |_{t = 0} = \int_{\Sigma}
  \delta_Y (H + \gamma u^{- 1} u_{\nu} - \mu) u^{\gamma} \phi \mathrm{d}
  \mathcal{H}^{n - 1} . \label{second variation non-explicit}
\end{equation}
\begin{definition}
  \label{def stable}We say that a warped $\mu$-bubble $\Omega = \Omega_0$ is
  stable if
  \begin{equation}
    \tfrac{\mathrm{d}^2}{\mathrm{d} t^2} E (\Omega_t) |_{t = 0} \geqslant 0.
    \label{eq stable via E}
  \end{equation}
  For convenience, we also say $\Sigma$ is stable if $\Sigma$ is of vanishing
  $H + \gamma u^{- 1} u_{\nu} - \mu$ and the right-hand side of \eqref{second
  variation non-explicit} is non-negative.
\end{definition}

By a direct calculation (see {\cite{antonelli-spectral-2024}},
{\cite{chodosh-stable-2024}}), the first variation of $\eta = H + \gamma u^{-
1} u_{\nu} - \mu$ is given
by{\foldedcomment{+wTScXpEcPNtGrq}{+wTScXpEcPNtGrr}{comment}{bk21}{1766907138}{}{If
$\eta \equiv 0$ along $\Sigma$, denote by $\mathcal{L}_0 \phi = \delta_{\phi
\nu} \eta$ where the more general $\mathcal{L}_{\mu}$ is defined in Lemma
\ref{first variation eta general}.

\begin{definition}
  \label{op stable}We say that a warped $\mu$-hypersurface is
  $\mathcal{L}_0$-stable if the principal eigenvalue of the operator
  $\mathcal{L}_0$ is non-negative.
\end{definition}

Recall that by the Krein-Rutman theorem ({\color{orange}{todo}} ),
$\mathcal{L}_0$ admits an eigenvalue, called the principal eigenvalue, which
has the least real part among all eigenvalues.

\begin{lemma}
  The two stability notions in Definition \ref{def stable} and \ref{op
  stable} are in fact equivalent. 
\end{lemma}

\begin{proof}
  First, the energy stability implies the $\mathcal{L}_0$-stability: let
  $\phi$ be the principal eigenfunction, so $\mathcal{L}_0 \phi = \lambda_1
  \phi$; clearly, the energy stability $\int_{\Sigma} u^{\gamma} \phi
  \mathcal{L}_0 \phi \geqslant 0$ which clearly gives $\lambda_1 \geqslant 0$.
  
  Secondly, 
\end{proof}}}
\begin{align}
& \delta_{\phi \nu} \eta \label{first var crude} \\
= & - \Delta_{\Sigma} \phi - (\ensuremath{\operatorname{Ric}}(\nu) + |A|^2)
\phi \\
& \quad - \gamma w_{\nu}^2 \phi + \gamma u^{- 1} \phi (\Delta_g u -
\Delta_{\Sigma} u - H u_{\nu}) - \gamma \langle \nabla w, \nabla^{\Sigma}
\phi \rangle - \mu_{\nu} \phi,
\end{align}
where $w = \log u$. We also call $\delta_{\phi \nu} \eta$ to be the
linearisation of $\eta$.

\subsection{Rewrite of the second variation}

Now we write the linearisation of $\eta$ or the second variation of the warped
$\mu$-bubble in a way that is related to \eqref{lambda}.

\begin{lemma}
  \label{first variation eta general}The first variation of $\eta = H + \gamma
  u^{- 1} u_{\nu} - \mu$ is $\delta_{\phi \nu} \eta =\mathcal{L}_{\mu} \phi$,
  where given a function $f$ on $\Sigma$, the elliptic operator
  $\mathcal{L}_f$ is given
  by{\foldedcomment{+2dvGGyL01y79rGV2}{+2dvGGyL01y79rGV3}{comment}{bk21}{1765255628}{}{\begin{align}
    & L_f \phi \nonumber\\
    = & - \Delta_{\Sigma} \phi - \gamma u^{- 1} \phi \Delta_{\Sigma} u -
    \gamma \langle \nabla w, \nabla^{\Sigma} \phi \rangle + \sigma \gamma |
    \nabla^{\Sigma} w|^2 \phi \nonumber\\
    & \quad - \tfrac{1}{2} |A^0 |^2 \phi - \tfrac{1}{n - 1} \eta (n f -
    \gamma w_{\nu}) \phi - \tfrac{n}{2 (n - 1)} f^2 \phi \nonumber\\
    & \quad + \tfrac{1}{2} R_{\Sigma} \phi \nonumber\\
    & \quad - \left( \frac{2 n (1 - \sigma) - (n - 1) \gamma}{2 (2 (n - 1) (1
    - \sigma) - (n - 2) \gamma)} \mu^2 + \mu_{\nu} + \Lambda \right) \phi
    \nonumber\\
    & \quad - \tfrac{2 (n - 1) (1 - \sigma) - (n - 2) \gamma}{2 (n - 1)}
    \gamma (w_{\nu} - \tfrac{1}{2 (n - 1) (1 - \sigma) - (n - 2) \gamma}
    \mu)^2 \phi, \nonumber
  \end{align}}}
  \begin{equation}
    \mathcal{L}_f \phi = L \phi + \tfrac{1}{2} R_{\Sigma} \phi + Z_f \phi + W
    \phi \label{LL f},
  \end{equation}
  with $L$, $Z_f$ and $W$ given by
  \begin{equation}
    L \phi = - \Delta_{\Sigma} \phi - \gamma u^{- 1} \phi \Delta_{\Sigma} u -
    \gamma \langle \nabla w, \nabla^{\Sigma} \phi \rangle, \label{L}
  \end{equation}
  and
  \begin{equation}
    Z_f = - \tfrac{1}{n - 1} \eta (n f - \gamma w_{\nu}) - \tfrac{n}{2 (n -
    1)} f^2, \label{Z}
  \end{equation}
  and
\begin{align}
W = & - \tfrac{1}{2} |A^0 |^2 - \left( \frac{2 n - (n - 1) \gamma}{2 (2 (n
- 1) - (n - 2) \gamma)} \mu^2 + \mu_{\nu} + \Lambda \right) \label{W}
\\
& \quad - \tfrac{2 (n - 1) - (n - 2) \gamma}{2 (n - 1)} \gamma (w_{\nu} -
\tfrac{1}{2 (n - 1) - (n - 2) \gamma} \mu)^2 .
\end{align}
\end{lemma}

\begin{proof}
  It follows from Schoen-Yau's rewrite {\cite{schoen-proof-1979}}, $|A|^2 =
  |A^0 |^2 + \tfrac{1}{n - 1} H^2$ that
\begin{align}
\ensuremath{\operatorname{Ric}} (\nu) + |A|^2 = & \tfrac{1}{2} R_g -
\tfrac{1}{2} R_{\Sigma} + \tfrac{1}{2} H^2 + \tfrac{1}{2} |A|^2
\\
= & \tfrac{1}{2} R_g - \tfrac{1}{2} R_{\Sigma} + \tfrac{n}{2 (n - 1)} H^2
+ \tfrac{1}{2} |A^0 |^2 .
\end{align}
  Here, $A^0$ is the traceless part of the second fundamental form. Since $H =
  \eta + \mu - \gamma w_{\nu}$, then
\begin{align}
& \ensuremath{\operatorname{Ric}} (\nu) + |A|^2 \\
= & \tfrac{1}{2} R_g - \tfrac{1}{2} R_{\Sigma} + \tfrac{1}{2} |A^0 |^2 +
\tfrac{n}{2 (n - 1)} \eta^2 + \tfrac{n}{n - 1} \eta (\mu - \gamma w_{\nu})
\\
& \quad + \tfrac{n}{2 (n - 1)} \mu^2 - \tfrac{n}{n - 1} \gamma \mu
w_{\nu} + \tfrac{n}{2 (n - 1)} \gamma^2 w_{\nu}^2 .
\end{align}
  Using the above and after suitable regrouping, we obtain
\begin{align}
& \delta_{\phi \nu} \eta \\
= & - \Delta_{\Sigma} \phi - \gamma u^{- 1} \phi \Delta_{\Sigma} u -
\gamma \langle \nabla w, \nabla^{\Sigma} \phi \rangle \\
& \quad - \tfrac{1}{2} |A^0 |^2 \phi - \tfrac{n}{n - 1} \eta (\mu -
\gamma w_{\nu}) \phi - \gamma \eta w_{\nu} - \tfrac{n}{2 (n - 1)} \eta^2
\phi \label{terms with eta} \\
& \quad - \tfrac{n}{2 (n - 1)} \mu^2 \phi + (\tfrac{n}{n - 1} \gamma -
\gamma) \phi \mu w_{\nu} + w_{\nu}^2 (- \tfrac{n}{2 (n - 1)} \gamma^2 +
\gamma^2 - \gamma) \phi \label{terms before completing the sqaure}
\\
& \quad + \tfrac{1}{2} R_{\Sigma} \phi - (- \gamma u^{- 1} \Delta_g u +
\tfrac{1}{2} R_g) \phi - \mu_{\nu} \phi .
\end{align}
  With some simplification of the terms on line \eqref{terms with eta},
  completing the square on the line \eqref{terms before completing the sqaure}
  and using the definitions of $\Lambda$, $Z_{\mu}$ in \eqref{Z} and $W$ in
  \eqref{W} finishes the proof.
\end{proof}

{}

\begin{lemma}
  \label{rewrite of L}Let $\psi = u^{\gamma / 2} \phi$ and $L$ be given in
  \eqref{L}, then
  \[ \int_{\Sigma} u^{\gamma} \phi L \phi = \tfrac{4}{4 - \gamma}
     \int_{\Sigma} | \nabla_{\Sigma} \psi |^2 - \gamma (1 - \tfrac{\gamma}{4})
     \int_{\Sigma} | \psi \nabla_{\Sigma} w - \tfrac{1}{2 (1 -
     \tfrac{\gamma}{4})} \nabla_{\Sigma} \psi |^2 . \]
\end{lemma}

\begin{proof}
  Inserting $\psi = u^{\gamma / 2} \phi$ in the definition \eqref{L} of $L$,
\begin{align}
u^{\gamma} \phi L \phi = & - u^{\gamma / 2} \psi \Delta_{\Sigma} (u^{-
\gamma / 2} \psi) - \gamma u^{- 1} \psi^2 \Delta_{\Sigma} u \\
& \quad - \gamma \langle \nabla w, \nabla^{\Sigma} (u^{- \gamma / 2}
\psi) \rangle u^{\gamma / 2} \psi .
\end{align}
  Take the integration on $\Sigma$, and by integration by parts on the first
  two terms and a direct calculation,
  \begin{equation}
    \int_{\Sigma} u^{\gamma} \phi L \phi = \int_{\Sigma} | \nabla_{\Sigma}
    \psi |^2 + \gamma \langle \nabla w, \nabla^{\Sigma} \psi \rangle +
    (\tfrac{\gamma^2}{4} - \gamma) | \nabla^{\Sigma} w| \psi^2 .
    \label{rewrite of L without squares}
  \end{equation}
  The rest is a simple completing of the squares.
\end{proof}

The following is an immediate corollary from Lemmas \ref{first variation eta
general} and \ref{rewrite of L}.

\begin{corollary}
  \label{second var warped h hypersurface}For a warped $\mu$-hypersurface
  $\Sigma$, the second variation \eqref{eq stable via E} can be rewritten as
\begin{align}
& \tfrac{\mathrm{d}^2}{\mathrm{d} t^2} E (\Omega_t) |_{t = 0} \\
= & \tfrac{4}{4 - \gamma} \int_{\Sigma} | \nabla_{\Sigma} \psi |^2 +
\tfrac{1}{2} \int_{\Sigma} R_{\Sigma} \psi^2 \\
& \quad - \gamma (1 - \tfrac{\gamma}{4}) \int_{\Sigma} | \psi
\nabla_{\Sigma} w - \tfrac{1}{2 (1 - \tfrac{\gamma}{4})} \nabla_{\Sigma}
\psi |^2 + \int_{\Sigma} W \psi^2,
\end{align}
  where $W$ is given in \eqref{W} and $u^{\gamma / 2} \phi = \psi$.
\end{corollary}

\subsection{Model metrics for Theorem \ref{spec llarull band}}

Now we discuss some model metrics where $\Lambda (g_{\xi}, u_{\xi})$ is a
constant. The model metrics are $g_{\xi} = \mathrm{d} t^2 + \xi (t)^2
g_{\mathbb{S}^{n - 1}}$ with
\begin{equation}
  \xi (t) = a \sin (b t), \label{eq xi from a b}
\end{equation}
where
\begin{align}
a (\gamma, \Lambda) & = \sqrt{\frac{\gamma (n - 1) (n - 2) (2 n - (n - 1)
\gamma)}{2 \Lambda (2 (n - 2) - (n - 3) \gamma)}}, \\
b (\gamma, \Lambda) & = \frac{(2 - \gamma) \sqrt{2 \Lambda}}{\sqrt{\gamma (2
n - (n - 1) \gamma) (2 (n - 1) - (n - 2) \gamma)}} .
\end{align}
The functions $u$, $h$ in $H + \gamma u^{- 1} u' - m = 0$ by
\begin{equation}
  u = \xi^{\tfrac{1}{2 - \gamma}}, \text{ } m = \frac{2 (n - 1) + \gamma (2 -
  n)}{2 - \gamma} \xi' \xi^{- 1} . \label{model u h}
\end{equation}
and $m$ satisfies the ODE
\begin{equation}
  m' + \tfrac{- n \gamma + \gamma + 2 n}{4 (n - 1) + 2 \gamma (2 - n)} m^2 +
  \Lambda = \tfrac{1}{2} (n - 1) (n - 2) \xi^{- 2} . \label{model ode}
\end{equation}
This metric also satisfies $- \gamma u^{- 1} \Delta_g u + \tfrac{1}{2} R_g =
\Lambda$. Here, $\Lambda > 0$, the case already appeared in
{\cite{chai-spectral-2024}}. Now for $\Lambda < 0$, then $\xi = a (\gamma, -
\Lambda) \sinh (b (\gamma, - \Lambda) t)$, and $\xi$ satisfies \eqref{model u
h} and \eqref{model ode}. Now for $\Lambda = 0$, $\xi = a (\gamma, 1) b
(\gamma, 1) t$ and $\xi$ satisfies \eqref{model u h} and \eqref{model ode}.

\section{Llarull type theorem for bands}\label{sec band}

In this section, we give the proof of Theorem \ref{spec llarull band}. Our
strategy is a classical foliation argument which makes the full use of the
rigidity of Listing's Theorem \ref{thm listing}.

\subsection{Construction of $\mu$}

Given a warped product metric $g_{\xi} = \mathrm{d} t^2 + \xi (t)^2
g_{\mathbb{S}^{n - 1}}$ where $\xi (t) > 0$ for all $t \in [t_-, t_+]$ and
$\gamma \geqslant 0$ on $\mathbb{S}_I^n,$ define
\[ h_{\xi} (t) = (n - 1) \xi' / \xi, \text{ } m_{\xi} (t) = (n - 1 +
   \tfrac{\gamma}{2 - \gamma}) \tfrac{\xi'}{\xi}, \text{ } u_{\xi} (t) =
   \xi^{\tfrac{1}{2 - \gamma}} . \]
Let $(M, g)$ be a band such that the map $f : (M, g) \to (\mathbb{S}_I^n,
g_{\xi})$ is of non-zero degree, define $\ensuremath{\operatorname{pr}}_I :
\mathbb{S}_I^n \to I$ projection,
\begin{equation}
  \mu := \mu_{\xi} = m_{\xi} \circ \ensuremath{\operatorname{pr}}_I \circ f :
  M \to \mathbb{R}, \label{mu}
\end{equation}
and
\begin{equation}
  f_{\Sigma} : \ensuremath{\operatorname{pr}}_{\mathbb{S}^{n - 1}} \circ f
  \circ i_{\Sigma} : \Sigma \to \mathbb{S}^{n - 1} \label{f Sigma}
\end{equation}
where $i_{\Sigma} : \Sigma \to M$ is the inclusion map of $\Sigma$ into $M$.

\subsection{Local rigidity}

\begin{proposition}
  \label{local rigidity}Assume that $(M, g)$ is given in Theorem \ref{spec
  llarull band}, and that $\Sigma$ is a stable warped $\mu$-hypersurface in
  $(M, g)$ disjoint from $\partial_{\pm} M$ and such that $f_{\Sigma}$ is of
  non-zero degree, then $\Sigma$ is a level set of
  $\ensuremath{\operatorname{pr}}_I \circ f$, $(\Sigma, g_{\Sigma})$ is
  isometric to $(\mathbb{S}^{n - 1}, (\xi \circ
  \ensuremath{\operatorname{pr}}_I \circ f)^2 g_{\mathbb{S}^{n - 1}})$ via
  $f_{\Sigma}$.
\end{proposition}

\begin{proof}
  First, we estimate $\mu_{\nu}$ appeared in the definition of $W$ in
  \eqref{W}. Recall that $\mu = m_{\xi} \circ \ensuremath{\operatorname{pr}}_I
  \circ f$, see \eqref{mu}. By the chain rule,
  \[ \mu_{\nu} = m_{\xi}' (\ensuremath{\operatorname{pr}}_I \circ f) \langle
     \nabla (\ensuremath{\operatorname{pr}}_I \circ f), \nu \rangle . \]
  Since $f$ is of distance non-increasing, $| \nabla
  (\ensuremath{\operatorname{pr}}_I \circ f) | \leqslant 1$; and since
  $m_{\xi}' < 0$,
  \begin{equation}
    \mu_{\nu} \geqslant m_{\xi}' (\ensuremath{\operatorname{pr}}_I \circ f) .
    \label{mu normal derivative}
  \end{equation}
  Now we can estimate the term $\frac{2 n - (n - 1) \gamma}{2 (2 (n - 1) - (n
  - 2) \gamma)} \mu^2 + \mu_{\nu} + \Lambda$ appeared in $W$, see \eqref{W}.
  We set
  \begin{equation}
    \Gamma = \frac{2 n - (n - 1) \gamma}{2 (2 (n - 1) - (n - 2) \gamma)}
    \label{Gamma}
  \end{equation}
  for convenience. By the comparison $\Lambda \geqslant f^{\ast}
  \Lambda_{g_{\xi}} = \Lambda_{g_{\xi}} \circ f$, $\mu = m_{\xi} \circ
  \ensuremath{\operatorname{pr}}_I \circ f$ and \eqref{mu normal derivative},
  we see
\begin{align}
& \Gamma \mu^2 + \mu_{\nu} + \Lambda \geqslant \Gamma m_{\xi}^2
(\ensuremath{\operatorname{pr}}_I \circ f) + m_{\xi}'
(\ensuremath{\operatorname{pr}}_I \circ f) + \Lambda_{g_{\xi}, u_{\xi}}
\circ f.
\end{align}
  It is easy to check that the right hand side of the above is $\tfrac{1}{2}
  (n - 1) (n - 2) \tfrac{1}{\xi^2 (\ensuremath{\operatorname{pr}}_I \circ
  f)}$. Hence,
  \[ \Gamma \mu^2 + \mu_{\nu} + \Lambda \geqslant \tfrac{1}{2} (n - 1) (n - 2)
     \tfrac{1}{\xi^2 (\ensuremath{\operatorname{pr}}_I \circ f)}, \]
  and evidently then
  \begin{equation}
    W \leqslant - \tfrac{1}{2} (n - 1) (n - 2) \tfrac{1}{\xi^2
    (\ensuremath{\operatorname{pr}}_I \circ f)} . \label{estimate W}
  \end{equation}
  Using the rewrite of Corollary \ref{second var warped h hypersurface} on the
  stability of $\Sigma$, and then \eqref{estimate W},
  \begin{equation}
    0 \leqslant \tfrac{4}{4 - \gamma} \int_{\Sigma} | \nabla_{\Sigma} \psi |^2
    + \tfrac{1}{2} \int_{\Sigma} R_{\Sigma} \psi^2 - \tfrac{1}{2} (n - 1) (n -
    2) \int_{\Sigma} \tfrac{1}{\xi^2 (\ensuremath{\operatorname{pr}}_I \circ
    f)} \psi^2 =: B_1 (\psi, \psi) \label{simplified stability for dim 3}
  \end{equation}
  for all $\psi \in C^2 (\Sigma)$. Note that, we have dropped the integrals of
  the non-negative terms listed below (up to constant multiples)
  \begin{equation}
    | \psi \nabla_{\Sigma} w - \tfrac{1}{2 (1 - \tfrac{\gamma}{4})}
    \nabla_{\Sigma} \psi |^2, \text{ } (w_{\nu} - \tfrac{1}{2 (n - 1) - (n -
    2) \gamma} \mu)^2 \psi^2, \text{ and } |A^0 |^2 \psi^2 .
    \label{non-negative dropped}
  \end{equation}
  Set
  \begin{equation}
    \tilde{L} = - \tfrac{4}{4 - \gamma} \Delta_{\Sigma} + \tfrac{1}{2}
    (R_{\Sigma} - (n - 1) (n - 2) \tfrac{1}{\xi^2
    (\ensuremath{\operatorname{pr}}_I \circ f)}) . \label{L tilde}
  \end{equation}
  Now we perform an analysis similar to
  {\cite{fischer-colbrie-structure-1980}} in three dimensions. Now $\tilde{L}
  = - \tfrac{4}{4 - \gamma} \Delta_{\Sigma} + \tfrac{1}{2} (R_{\Sigma} -
  \tfrac{2}{\xi^2 (\ensuremath{\operatorname{pr}}_I \circ f)})$. We set $\psi
  = 1$ and use the Gauss-Bonnet theorem in \eqref{simplified stability for dim
  3}, and obtain
  \begin{equation}
    2 \pi - \int_{\Sigma} \tfrac{1}{\xi^2 (\ensuremath{\operatorname{pr}}_I
    \circ f)} \mathrm{d} A_g \geqslant B_1 (1, 1) = 2 \pi \chi (\Sigma) -
    \int_{\Sigma} \tfrac{1}{\xi^2 (\ensuremath{\operatorname{pr}}_I \circ f)}
    \mathrm{d} A_g \geqslant 0. \label{after gauss-bonnet}
  \end{equation}
  where $\mathrm{d} A_g$ denotes the area element of $\Sigma$ with respect to
  the metric $g$. Since $g \geqslant g_{\xi}$, $\mathrm{d} A_g \geqslant
  f^{\ast} (\xi^2 \mathrm{d} A_{\mathbb{S}^2})$, here $A_{\mathbb{S}^2}$ is
  the area element of the standard 2-sphere. So the left hand side of
  \eqref{after gauss-bonnet} is non-positive, which forces all inequalities of
  \eqref{after gauss-bonnet} have to be equalities, in particular, $B_1 (1, 1)
  = 0$.
  
  It is easy to check $B_1 (\psi, \psi) = \int_{\Sigma} \psi \tilde{L}
  \psi$. Hence by \eqref{simplified stability for dim 3}, the first eigenvalue
  $\lambda_1$ of $\tilde{L}$ is non-negative. In fact, since $B_1 (1, 1) = 0$,
  $\lambda_1 = 0$ and hence $\tilde{L} 1 = 0$, that is, $R_{\Sigma} =
  \tfrac{2}{\xi^2 (\ensuremath{\operatorname{pr}}_I \circ f)}$ along $\Sigma$.
  Tracing back the equalities in the full stability inequality (see Corollary
  \ref{second var warped h hypersurface}), and noting that we have used the
  non-negativity of the terms in \eqref{non-negative dropped}, we see
  \begin{equation}
    \nabla_{\Sigma} w = 0, \text{ } w_{\nu} - \tfrac{1}{4 - \gamma} \mu = 0,
    \text{ and } |A^0 | = 0 \text{ along } \Sigma \label{non-negative is zero
    in dim 3}
  \end{equation}
  (note that $n = 3$). Moreover, $\nabla (\ensuremath{\operatorname{pr}}_I
  \circ f) = \nu$ and $\Lambda = f^{\ast} \Lambda_{g_{\xi}}$ along $\Sigma$.
  
  Now we handle four dimensions. We use the estimate
  \begin{equation}
    \|(\mathrm{d} f_{\Sigma})_p \|^2 \leqslant \tfrac{1}{\xi^2
    (\ensuremath{\operatorname{pr}}_I (f (p)))} \label{eq estimate of
    gradient}
  \end{equation}
  in the estimate \eqref{estimate W} of $W$, and we obtain (using also $n =
  4$)
  \begin{equation}
    0 \leqslant \tfrac{4}{4 - \gamma} \int_{\Sigma} | \nabla_{\Sigma} \psi |^2
    + \tfrac{1}{2} \int_{\Sigma} (R_{\Sigma} - 6\|(\mathrm{d} f_{\Sigma})_p
    \|^2) \psi^2 =: B_2 (\psi, \psi) \label{simplified stability for dim 4}
  \end{equation}
  instead of \eqref{simplified stability for dim 3}. Let $\hat{L} = -
  \tfrac{4}{4 - \gamma} \Delta_{\Sigma} + \tfrac{1}{2} (R_{\Sigma} -
  6\|(\mathrm{d} f_{\Sigma})_p \|^2)$, by the above, the first eigenvalue of
  $\hat{L}$ is non-negative. We choose the corresponding eigenfunction $v$ to
  be positive that is, $\hat{L} v = \lambda_1 v \geqslant 0$. Without loss of
  generality, we assume that $\sup_{\Sigma} v = 1$.
  
  Let $\alpha = \tfrac{1}{4 - \gamma}$ and $g_{\Sigma}$ be the induced metric
  on $\Sigma$ from $g$. We consider the conformal metric $v^{4 \alpha}
  g_{\Sigma}$, and the scalar curvature of $\Sigma$ with respect to $v^{4
  \alpha} g_{\Sigma}$ is
\begin{align}
R_{\Sigma} (v^{4 \alpha} g_{\Sigma}) & = v^{- 4 \alpha} (R_{\Sigma} - 8
\alpha v^{- 1} \Delta_{\Sigma} v - 8 \alpha (\alpha - 1) | \nabla^{\Sigma}
\log v|^2) \\
& = v^{- 4 \alpha} (2 \lambda_1 + 6\|(\mathrm{d} f_{\Sigma})_p \|^2 - 8
\alpha (\alpha - 1) | \nabla^{\Sigma} \log v|^2) \\
& = v^{- 4 \alpha} (2 \lambda_1 - 8 \alpha (\alpha - 1) | \nabla^{\Sigma}
\log v|^2) + 6 \|(\mathrm{d} f_{\Sigma})_p \|^2_{v^{4 \alpha} g_{\Sigma}}
\\
& \geqslant 6 \|(\mathrm{d} f_{\Sigma})_p \|^2_{v^{4 \alpha} g_{\Sigma}}
\label{sign from listing}
\end{align}
  by $\lambda_1 \geqslant 0$ and the fact $0 < \alpha < 1$ from the
  assumptions. By the rigidity statement of Theorem \ref{thm listing} in
  dimension $n - 1 = 3$, the above has to be an equality, and then $\lambda_1
  = 0$ and $v = 1$. In particular, $f_{\Sigma} : (\Sigma, g_{\Sigma}) \to
  (\mathbb{S}^3, c^2 g_{\mathbb{S}^3})$ where the map $f_{\Sigma}$ is defined
  in \eqref{f Sigma} is an isometry for some constant $c > 0$. We will
  determine the constant $c$ shortly.
  
  As in dimension 3, we can also trace back the equalities: letting $\psi = 1$
  in \eqref{simplified stability for dim 4} first give $B_2 (1, 1) = 0$ and
  then in Corollary \ref{second var warped h hypersurface}, which forces
  \begin{equation}
    \nabla_{\Sigma} w = 0, \text{ } w_{\nu} - \tfrac{1}{6 - 2 \gamma} \mu = 0,
    \text{ and } |A^0 | = 0 \text{ along } \Sigma \label{local rigidity for u}
  \end{equation}
  similar to \eqref{non-negative is zero in dim 3}. Also, the equality of
  \eqref{eq estimate of gradient} is achieved, in particular, it implies $c =
  \xi \circ \ensuremath{\operatorname{pr}}_I \circ f$. As before, $\nabla
  (\ensuremath{\operatorname{pr}}_I \circ f) = \nu$ and $\Lambda = f^{\ast}
  \Lambda_{g_{\xi}}$ along $\Sigma$.
\end{proof}

\begin{remark}
  \label{non-positive eigenvalue}We have also shown that the first eigenvalue
  of the operator $\tilde{L}$ defined in \eqref{L tilde} is non-positive for
  any hypersurface $\Sigma$ such that $f_{\Sigma}$ is non-zero degree: assume
  otherwise, then it is a contradiction to the Gauss-Bonnet theorem in
  dimension $n = 3$ and to Theorem \ref{thm listing} due to Listing in
  dimension $n = 4$.
\end{remark}

\subsection{Analysis of the foliation}

The most important property of $\Sigma$ in Lemma \ref{local rigidity} is that
fact that the linearisation of $H + \gamma u^{- 1} u_{\nu} - \mu$ gives only
$- \Delta_{\Sigma}$ which allows us to construct a foliation with the property
that $H + \gamma u^{- 1} u_{\nu} - \mu$ is constant along every leaf of the
foliation.

\begin{proposition}[cf. {\cite[Lemma 3.4]{chai-band-2025}}]
  \label{existence foliation}Let $\Sigma$ be a hypersurface such that
  $\delta_Y (H + \gamma u^{- 1} u_{\nu} - \mu) = - \Delta_{\Sigma} \phi$, then
  there exists a foliation $\{\Sigma_t \}_{t \in (- \varepsilon,
  \varepsilon)}$ such that $\Sigma_0 = \Sigma$ and that $H + \gamma u^{- 1}
  u_{\nu} - \mu$ is constant along $\Sigma_t$.
\end{proposition}

The following is our key estimate of the foliation.

\begin{proposition}
  \label{sign of leaf}Let $\Sigma$ be given in Proposition \ref{local
  rigidity}, and let $\{\Sigma_t \}_{t \in (- \varepsilon, \varepsilon)}$ be
  the foliation constructed in Proposition \ref{existence foliation} such that
  $\Sigma_0 = \Sigma$. Let $\eta (t) = H + \gamma u^{- 1} u_{\nu} - \mu$, then
  $\eta (t) \leqslant 0$ if $t \in (0, \varepsilon)$ and $\eta (t) \geqslant
  0$ if $t \in (- \varepsilon, 0)$.
\end{proposition}

\begin{proof}
  Using the variational formula \eqref{first var crude},
  \begin{equation}
    \eta' (t) =\mathcal{L}_{\mu, t} \phi_t = L_t \phi_t + \tfrac{1}{2}
    R_{\Sigma_t} \phi_t + Z_{\mu, t} \phi_t + W_t \phi_t \label{variation eta
    t}
  \end{equation}
  where $\phi_t = \langle \partial_t, \nu_t \rangle > 0$. We use the subscript
  $t$ to indicate that the quantities live in $\Sigma_t$. Note that $Z_{\mu,
  t} \leqslant - \tfrac{1}{n - 1} \eta (t) (n \mu - \gamma w_{\nu_t}) = : -
  q_t \eta (t)$ by \eqref{Z} and then by \eqref{estimate W},
  \[ \eta' (t) \leqslant L_t \phi_t + \tfrac{1}{2} (R_{\Sigma_t} -
     \tfrac{2}{\xi^2 (\ensuremath{\operatorname{pr}}_I (f (p)))}) \phi_t - q_t
     \eta (t) \phi_t . \]
  Now we divide the above by $\phi_t$ and multiply by $\psi^2$, integrate on
  $\Sigma_t$ and obtain that
  \begin{equation}
    \eta' (t) \int_{\Sigma_t} \phi_t^{- 1} \psi^2 \leqslant \int_{\Sigma_t}
    \psi^2 \phi_t^{- 1} L_t \phi_t + \tfrac{1}{2} (R_{\Sigma_t} -
    \tfrac{2}{\xi^2 (\ensuremath{\operatorname{pr}}_I (f (p)))}) \psi^2 - \eta
    (t) \int_{\Sigma_t} q_t \psi^2 . \label{eta ode first}
  \end{equation}
  Let $\phi_t = u^{- \gamma / 2} e^{\zeta_t}$, then by a direct calculation
\begin{align}
\phi_t^{- 1} L_t \phi_t = & - \phi_t \Delta_{\Sigma_t} \phi_t - \gamma
u^{- 1} \Delta_{\Sigma} u - \gamma \phi_t^{- 1} \langle \nabla w,
\nabla^{\Sigma_t} \phi_t \rangle \\
= & - | \nabla^{\Sigma_t} \zeta_t |^2 - \Delta_{\Sigma_t} \zeta_t +
(\tfrac{\gamma^2}{4} - \gamma) | \nabla^{\Sigma_t} w|^2 -
\tfrac{\gamma}{2} \Delta_{\Sigma_t} w.
\end{align}
  Now we estimate $(- | \nabla^{\Sigma_t} \zeta_t |^2 - \Delta_{\Sigma_t}
  \zeta_t) \psi^2$ and $- \tfrac{\gamma}{2} \psi^2 \Delta_{\Sigma_t} w$ as
\begin{align}
& (| \nabla^{\Sigma_t} \zeta_t |^2 + \Delta_{\Sigma_t} \zeta_t) \psi^2
\\
= & | \nabla^{\Sigma_t} \zeta_t |^2 \psi^2 - 2 \langle \nabla^{\Sigma_t}
\psi, \nabla^{\Sigma_t} \zeta_t \rangle
+\ensuremath{\operatorname{div}}_{\Sigma_t} (\psi^2 \nabla^{\Sigma_t}
\zeta_t) \\
\geqslant & - | \nabla^{\Sigma_t} \psi |^2
+\ensuremath{\operatorname{div}}_{\Sigma_t} (\psi^2 \nabla^{\Sigma_t}
\zeta_t)
\end{align}
  by Cauchy-Schwarz inequality; and
  \[ \tfrac{\gamma}{2} \Delta_{\Sigma_t} w = \tfrac{\gamma}{2}
     \ensuremath{\operatorname{div}}_{\Sigma_t} (\psi^2 \nabla^{\Sigma_t} w) -
     \gamma \psi \langle \nabla^{\Sigma_t} \psi, \nabla^{\Sigma_t} w \rangle .
  \]
  With these two estimates in \eqref{eta ode first} and using the divergence
  theorem, we see
\begin{align}
& \eta' (t) \int_{\Sigma_t} \phi_t^{- 1} \psi^2 \\
\leqslant & \int_{\Sigma_t} (| \nabla^{\Sigma_t} \psi |^2 + \gamma \psi
\langle \nabla^{\Sigma_t} \psi, \nabla^{\Sigma_t} w \rangle +
(\tfrac{\gamma^2}{4} - \gamma) | \nabla^{\Sigma_t} w|^2 \psi^2)
\label{similar to rewrite L without squares} \\
& \qquad + \tfrac{1}{2} \int_{\Sigma_t} (R_{\Sigma_t} - \tfrac{2}{\xi^2
(\ensuremath{\operatorname{pr}}_I (f (p)))}) \psi^2 - \eta (t)
\int_{\Sigma_t} q_t \psi^2 .
\end{align}
  Note that this has the same form as \eqref{rewrite of L without squares}.
  Applying Lemma \ref{rewrite of L}, we obtain that
  \[ \eta' (t) \int_{\Sigma_t} \phi_t^{- 1} \psi^2 \leqslant \tfrac{4}{4 -
     \gamma} \int_{\Sigma} (| \nabla_{\Sigma} \psi |^2 + \tfrac{1}{2}
     (R_{\Sigma_t} - \tfrac{2}{\xi^2 (\ensuremath{\operatorname{pr}}_I (f
     (p)))}) \psi^2) - \eta (t) \int_{\Sigma_t} q_t \psi^2 . \]
  Recall that in the proof of Lemma \ref{local rigidity}, we have shown that
  the first eigenvalue of
  \[ \tilde{L}_t = - \tfrac{4}{4 - \gamma} \Delta_{\Sigma_t} + \tfrac{1}{2}
     (R_{\Sigma_t} - (n - 1) (n - 2) \tfrac{1}{\xi^2
     (\ensuremath{\operatorname{pr}}_I \circ f)}) \]
  is non-positive, see Remark \ref{non-positive eigenvalue}. Note that this is
  for every $t \in (- \varepsilon, \varepsilon)$. Let $\psi$ be replaced by
  the first eigenfunction $\psi_t$ of $\tilde{L}_t$, then
  \[ \eta' (t) \int_{\Sigma_t} \phi_t^{- 1} \psi_t^2 \leqslant - \eta (t)
     \int_{\Sigma_t} q_t \psi^2_t \]
  which gives
  \begin{equation}
    \tfrac{\mathrm{d}}{\mathrm{d} t} \left( \eta (t) \exp (\int^t Q (s)
    \mathrm{d} s) \right) \leqslant 0 \label{ode for eta}
  \end{equation}
  where $Q (t) = \frac{\int_{\Sigma_t} q_t \psi^2_t}{\int_{\Sigma_t} \phi_t^{-
  1} \psi_t^2}$. The above ODE and that $\eta (0) = 0$ finish the proof of the
  proposition.
\end{proof}

\begin{remark}
  The trick of obtaining \eqref{ode for eta} is borrowed from
  {\cite{chai-band-2025}}.
\end{remark}

Now we are ready to finish the proof of Theorem \ref{spec llarull band}.

\begin{proof}[Proof of Theorem \ref{spec llarull band}]
  By an argument of Chai-Sun {\cite[Appendix B]{chai-band-2025}} (replacing
  the $\mu$-hypersurface with the warped $\mu$-hypersurface), we can assume
  that there exists a non-trivial minimiser $\Omega$ to the functional
  \eqref{eq E band}. Let $\Sigma' = \partial \Omega \backslash \partial_- M$,
  then by {\cite[Lemma 6.3]{rade-scalar-2023}}, there exists a connected
  component $\Sigma$ of $\Sigma'$ such that $f_{\Sigma} : \Sigma \to
  \mathbb{S}^{n - 1}$ has non-zero degree. Moreover, $\Sigma$ is a stable
  warped $\mu$-hypersurface. By Lemma \ref{local rigidity}, the linearisation
  of $H + \gamma u^{- 1} u_{\nu} - \mu$ along $\Sigma$ is just $-
  \Delta_{\Sigma}$ and it leads to the existence of a foliation $\{\Sigma_t
  \}_{t \in (- \varepsilon {,} \varepsilon)}$ such that $\Sigma_0 = \Sigma$ by
  Proposition \ref{existence foliation}. By Proposition \ref{sign of leaf},
  $\eta (t) \leqslant 0$ if $t \in (0, \varepsilon)$ and $\eta (t) \geqslant
  0$ if $t \in (- \varepsilon, 0)$.
  
  For each non-zero $t \in (- \varepsilon, \varepsilon)$, define $\Omega_t$ as
  follows
  \[ \Omega_t = \Omega \cup \tilde{\Omega}_t \text{ if } t > 0, \text{ }
     \Omega \backslash \tilde{\Omega}_t \text{ if } t < 0, \]
  where $\tilde{\Omega}_t$ is the region bounded by $\Sigma_t$ and $\Sigma$.
  
  We consider $t > 0$ and the case $t < 0$ is similar. By the first variation
  of $E$,
  \[ \tfrac{\mathrm{d}}{\mathrm{d} t} E (\Omega_t) = \int_{\Sigma_t} (H +
     \gamma u^{- 1} u_{\nu} - \mu) u^{\gamma} \phi_t = \eta (t)
     \int_{\Sigma_t} u^{\gamma} \phi_t \leqslant 0 \]
  where $Y = \phi_t \nu$. So, each $\Omega_t$ is also a minimiser to the
  functional \eqref{eq E band}, and each $\Sigma_t$ satisfies the conclusions
  of Lemma \ref{local rigidity}. Hence, by connectedness, we find that $M$ is
  globally foliated by such hypersurfaces. Since each level set is isometric
  to $(\xi \circ \ensuremath{\operatorname{pr}}_I \circ f)^2 g_{\mathbb{S}^{n
  - 1}}$ and $\nabla (\ensuremath{\operatorname{pr}}_I \circ f) = \nu$, we see
  the metric is $\mathrm{d} t^2 + \xi^2 (t) g_{\mathbb{S}^{n - 1}}$, and the
  function $u$ is a constant multiple of $u^{\tfrac{1}{2 - \gamma}}$ follows
  from \eqref{local rigidity for u}.
\end{proof}

\begin{remark}
  \label{slightly general}The proof also works if we replace the comparisons
  of the mean curvatures by $H + \gamma u^{- 1} u_{\nu} \geqslant \mu$ on
  $\partial_+ M$ and $H + \gamma u^{- 1} u_{\nu} \leqslant \mu$ on $\partial_-
  M$. In this case, we can obtain additionally that $\partial_{\pm} M$ are
  both level sets of the function $\ensuremath{\operatorname{pr}}_I \circ f$.
\end{remark}

\section{Llarull's theorem with conical end points}\label{sec cone}

In this section, we give the proof for Theorems \ref{llarull cone theorem} and
\ref{cone
llarull}.{\foldedcomment{+2WO8POgj2162sL3I}{+2WO8POgj2162sL3J}{comment}{bk21}{1767518938}{}{{\tt
  \trivlist{\item[\color{rgb:black,10;red,9;green,4;yellow,2}{\color{red}{(\%i
  16)
  }}]\mbox{}{\color{blue!50!black}f(a):=(a*g*(a+n-2-s*a)+(n-1)*(n-2)/2)/(n-1+a*g)\^{}2}\item[]\mbox{}\
  
  $\text{\text{{\ttfamily{{\color{red}{(\%o16) }}}}}} f (a) := \frac{ag (a +
  n - 2 + (- s) a) + \frac{(n - 1)  (n - 2)}{2}}{(n - 1 + ag)^2}$}
  \trivlist{\item[\color{rgb:black,10;red,9;green,4;yellow,2}{\color{red}{(\%i
  18)
  }}]\mbox{}{\color{blue!50!black}diff(f(a),a),factor}\item[]\mbox{}$\text{\text{{\ttfamily{{\color{red}{(\%o18)
  }}}}}} - \left( \frac{ag (2 ns - 2 s + gn - 2 n - 2 g + 2)}{(n + ag - 1)^3}
  \right)$}
}}}

\subsection{Tangent cone analysis}

Let $p \in f^{- 1} (p_{\xi})$ where $p_{\xi}$ is the conical point of the
metric $g_{\xi}$ at $t = t_-$. Without loss of generality, we assume that $t_-
= 0$. In this subsection, we assume that the tangent cone at $p$ of the metric
$g$ and the tangent cone at $p_{\xi}$ of the metric $g_{\xi}$ are isometric.
In other words, near $p$, there exists a coordinate system such that the
metric is of the form $\mathrm{d} t^2 + A^2 t^2 g_{\mathbb{S}^{n - 1}} + g_1$
with $A$ satisfying \eqref{unit cross-section size} and $g_1$ is a 2-form
which is small compared to the metric $\mathrm{d} t^2 + A^2 t^2
g_{\mathbb{S}^{n - 1}}$.

First, we prove a spectral version of Theorem \ref{thm listing}.

\begin{theorem}
  \label{spectral listing theorem}Let $0 \leqslant \gamma < \tfrac{2 (n -
  1)}{n - 2}$ and $f : (M^n, g) \to (\mathbb{S}^n, g_{\mathbb{S}^n})$ be a
  spin map of non-zero degree and $u$ a positive smooth function on $M$ such
  that
  \[ \Lambda_{g, u} \geqslant \tfrac{1}{2} n (n - 1) \tfrac{(f^{\ast}
     g_{\mathbb{S}^n}) (X, X)}{g (X, X)} \]
  at all points $p$ and for all non-zero vectors $X \in T_p M$, then there
  exists a positive constant $c$ such that $f : (M, g) \to (\mathbb{S}^n, c
  g_{\mathbb{S}^n})$ is an isometry and $u$ is a constant.
\end{theorem}

\begin{proof}
  The theorem follows from Listing's Theorem \ref{thm listing} by a conformal
  change. Let $f_1 = u^{\beta}$ where $4 (n - 1) / (n - 2) \beta = 2 \gamma$
  and $\hat{g} = f_1^{\tfrac{4}{n - 2}} g$. By the assumption, $1 > \beta$. We
  use the well known formula of conformal change
  \[ R_{\hat{g}} = f_1^{- \tfrac{4}{n - 2}} (R_g - \tfrac{4 (n - 1)}{n - 2}
     f^{- 1} \Delta f) . \]
  We have the estimate using that $1 > \beta$,
\begin{align}
R_{\hat{g}} = & f_1^{- \tfrac{4}{n - 2}} (R_g - \tfrac{4 (n - 1)}{n - 2}
(\beta \tfrac{\Delta_g u}{u} + \beta (\beta - 1) \tfrac{| \nabla
u|^2}{u^2})) \\
= & f_1^{- \tfrac{4}{n - 2}} (R_g - 2 \gamma u^{- 1} \Delta_g u - 2 \gamma
(\beta - 1) \tfrac{| \nabla u|^2}{u^2}) \\
= & f_1^{- \tfrac{4}{n - 2}} (2 \Lambda_{g, u} - 2 \gamma (\beta - 1)
\tfrac{| \nabla u|^2}{u^2}) \\
\geqslant & f_1^{- \tfrac{4}{n - 2}} n (n - 1) \tfrac{(f^{\ast}
g_{\mathbb{S}^n}) (X, X)}{g (X, X)} \label{general listing ineq}
\\
= & n (n - 1) \tfrac{(f^{\ast} g_{\mathbb{S}^n}) (X, X)}{\hat{g} (X, X)} .
\end{align}
  By Listing's Theorem \ref{thm listing} for the metric $\hat{g}$, $\hat{g}$
  is isometric to some $c g_{\mathbb{S}^n}$ for some positive constant $c_1$.
  Then the inequality \eqref{general listing ineq} must be an identity, so $u$
  must be a constant and $f_1$ is also a constant. Hence, the theorem is
  proved.
\end{proof}

\begin{remark}
  We can apply the same trick and replace $\tfrac{(f^{\ast} g_{\mathbb{S}^n})
  (X, X)}{g (X, X)}$ with the square root of $\tfrac{(f^{\ast}
  g_{\mathbb{S}^n}) (X \wedge Y, X \wedge Y)}{g (X \wedge Y, X \wedge Y)}$
  where $X$ and $Y$ are vectors such that $X \wedge Y \neq 0$. Also, we can
  replace $(\mathbb{S}^n, g_{\mathbb{S}^n})$ with a manifold $(N, \bar{g})$
  with positive curvature operator and $n (n - 1)$ with $R_{\bar{g}} \circ f
  $. See {\cite[Theorem 1]{listing-scalar-arxiv-2010}}.
\end{remark}

Denote
\begin{equation}
  \lambda (g_{\Sigma}, v) = - \gamma v^{- 1} \Delta_{\Sigma} v + \tfrac{1}{2}
  R_{\Sigma}, \label{Lambda level}
\end{equation}
where $v$ is a function on $\Sigma$. We study the relation of $\lambda
(g_{\Sigma}, v)$ and $\Lambda_{\bar{g}, \bar{u}}$ where $\Sigma$ is the unit
cross-section of the cone $(C_1 = (0, \infty) \times \Sigma, \bar{g} =
\mathrm{d} t^2 + t^2 g_{\Sigma})$ for some metric $g_{\Sigma}$ and $\bar{u} =
t^{\alpha} v$. Let $\Sigma_t = \{t\} \times \Sigma$, using the first variation
formula \eqref{first var crude} on $\Sigma_t$, we see
\begin{align}
- \tfrac{n - 1}{t^2} - \tfrac{\alpha \gamma}{t^2} = & \partial_t (H + \gamma
u^{- 1} \partial_t u) \\
= & - \gamma u^{- 1} \Delta_g u + \tfrac{1}{2} R_g \\
& \quad - (- \gamma u^{- 1} \Delta_{\Sigma_t} u  + \tfrac{1}{2}
R_{\Sigma_r}) \\
& \quad + \tfrac{1}{2} \tfrac{n}{n - 1} H_{\Sigma_t}^2 + \gamma
(\tfrac{\partial_t u}{u})^2 + \gamma H_{\Sigma_t} \tfrac{\partial_t u}{u} .
\end{align}
Considering that $\tfrac{| \nabla^g u|^2}{u^2} = \tfrac{| \nabla^{\Sigma}
v|^2}{v^2} + \alpha^2 t^{- 2}$, $H_{\Sigma_t} = \tfrac{n - 1}{t}$, the
definitions of $\Lambda_{g, u}$ and $\lambda_{\Sigma, v}$, we have
\begin{equation}
  \lambda (g_{\Sigma}, v) = \Lambda_{g, u} + \alpha \gamma (\alpha + n - 2)
  t^{- 2} + \tfrac{1}{2} (n - 1) (n - 2) t^{- 2} . \label{spectral descent}
\end{equation}
Now we show that the exponent in the asymptotics of $u$ is $\tfrac{1}{2 -
\gamma}$.

\begin{lemma}
  Let $u$ be given in Theorem \ref{llarull cone theorem}, then $u (t) =
  t^{\tfrac{1}{2 - \gamma}} (1 + O (t))$.
\end{lemma}

\begin{proof}
  Let $\Sigma_t$ be the $t$-level set given by the warped product metric.
  First, if $\alpha < \bar{\alpha} := \tfrac{1}{2 - \gamma}$, then
  $H_{\Sigma_t} = (n - 1 + \alpha \gamma) t^{- 1} + o (t^{- 1})$ which implies
  that $\Sigma_t$ is a strict barrier, but by Theorem \ref{spec llarull band},
  this is impossible, so $\alpha \geqslant \bar{\alpha}$. By the comparisons
  $\Lambda_{g, u} \geqslant \Lambda_{g_{\xi}, u_{\xi}} \circ f$, and that $f$
  is distance non-increasing, then these properties were preserved by taking
  the tangent cones. Also, \eqref{spectral descent} holds up to some error
  term of order $O (t^{- 1})$. By $\alpha \geqslant \bar{\alpha}$ and the
  comparison $\Lambda_{g, u} \geqslant \Lambda_{g_{\xi}, u_{\xi}} \circ f$ in
  the limit, it follows that
  \[ \lambda (g_{\Sigma}, v) \geqslant \lambda (g_{\Sigma}, 1) . \]
  By Theorem \ref{spectral listing theorem}, $v$ is a positive constant, in
  particular, the equality holds in the above which implies $\alpha =
  \bar{\alpha}$ by \eqref{spectral descent}.
\end{proof}

\begin{proposition}
  \label{foliation naer conical point}There exists a foliation $\{\Sigma_t
  \}_{t \in (0, \varepsilon)}$ such that each leaf has constant $\eta = H +
  \gamma u^{- 1} u_{\nu} - h$.
\end{proposition}

Given $t > 0$, we consider $\hat{g}_t (\tau, x) = t^{- 2} g (t \tau, x) =
(\mathrm{d} \tau^2 + A^2 \tau^2 g_{\mathbb{S}^{n - 1}} + g_1 t^{- 2})$ and
$\hat{u}_t (\tau, x) = t^{- \tfrac{1}{2 - \gamma}} u (t \tau, x)$ with $\tau
\in (0, 1]$ and $x \in \mathbb{S}^{n - 1}$.

Let $\hat{\nu}_{\tau}$ be the unit normal of the $\tau$-level set,
$\hat{H}_{\tau} (x)$ denotes the mean curvature of $\tau$-level set
$\hat{\Sigma}_{\tau}$ with respect to the metric $\hat{g}$. It is clear that
$\hat{H}_{\tau} (x) = t H_{t \tau} (x)$ where $H_t (x)$ is the mean curvature
of the $t$-level set with respect to the metric $g$.

Set
\begin{equation}
  \hat{\eta}_{t, \tau} (x) = \hat{H}_{\tau} (x) + \gamma (\log
  \hat{u})_{\hat{\nu}} - \hat{h} . \label{eta hat t parameter}
\end{equation}
Note that all of the above quantities depends on $t$ and we have made some
dependence implicit. Now we consider only the $\tau$-level set
$\hat{\Sigma}_1$ and the perturbations $\hat{\Sigma}_{\tau, v}$ defined by
\[ (\tau + v (x), x) \]
where $v : \mathbb{S}^{n - 1} \to \mathbb{R}$ is a function. Denote every
geometric quantity on $\hat{\Sigma}_{\tau, v}$ with a subscript $t, \tau, v$.
Using this setup, Proposition \ref{foliation naer conical point} is equivalent
to the following.

\begin{proposition}
  For a sufficiently small $\varepsilon$, for each $t \in (0, \varepsilon)$
  there exists a $v (x, t)$, $x \in \mathbb{S}^{n - 1}$ such that
  $\hat{\Sigma}_{1, t v (\cdot, t)}$ has constant $\hat{\eta}$.
\end{proposition}

\begin{proof}
  By the first variation \eqref{first var crude} and the Taylor expansion of
  $\hat{\eta}_{1, t v}$ with respect to $t$,
\begin{align}
& (\hat{\eta}_{t, 1, t v} - \hat{\eta}_{t, 1, 0}) t^{- 1} \\
= & - \Delta_t v - | \hat{A} |_{\hat{g}}^2 v
-\ensuremath{\operatorname{Ric}}_{\hat{g}} (\hat{\nu}_1, \hat{\nu}_1) v -
\gamma \hat{u}^{- 2} \hat{u}_{\hat{\nu}}^2 v \\
& \quad + \gamma \hat{u}^{- 1} (\Delta_{\hat{g}} \hat{u} - \Delta_t
\hat{u} - \hat{H}_1 \hat{u}) v - \gamma \hat{u}^{- 1} \langle \hat{\nabla}
\hat{u}, \hat{\nabla} v \rangle \\
& \qquad - \langle \hat{\nabla} \hat{h}, \hat{\nu}_1 \rangle v +
\hat{\partial}_{\tau}^{\top} \hat{\eta} + O (t) .
\end{align}
  By the convergence of $\hat{g}$ to $\mathrm{d} \tau^2 + A^2 \tau^2
  g_{\mathbb{S}^{n - 1}}$ and $\hat{u}$ to $\tau^{\tfrac{1}{2 - \gamma}}$,
  \begin{equation}
    (\hat{\eta}_{t, 1, t v} - \hat{\eta}_{t, 1, 0}) t^{- 1} = - \bar{\Delta}_1
    v + O (t) . \label{taylor reduced}
  \end{equation}
  By the convergence, $\hat{\eta}_{t, 1, 0} = O (t^2)$ and hence
  \[ \hat{\eta}_{t, 1, t v} t^{- 1} = - \bar{\Delta}_1 v + O (t) . \]
  Due to the convergence as $t \to 0$,
  \begin{equation}
    \hat{\eta}_{t, 1, t v} t^{- 1} = - \bar{\Delta}_1 v + O (t) . \label{limit
    perturb}
  \end{equation}
  Fix a sufficiently small positive number $\varepsilon$, and define the
  mapping
  \begin{equation}
    F (t, v) = \tfrac{\hat{\eta}_{t, 1, t v}}{t} - | \hat{\Sigma}_1 |^{- 1}
    \int_{\hat{\Sigma}_1} \tfrac{\hat{\eta}_{t, 1, t v}}{t}, \label{F t v}
  \end{equation}
  where the integration is with respect to the metric $\mathrm{d} \tau^2 + A^2
  \tau^2 g_{\mathbb{S}^{n - 1}}$ and $| \hat{\Sigma}_1 |$ is the volume of
  $\hat{\Sigma}_1$ with respect to the metric $\mathrm{d} \tau^2 + A^2 \tau^2
  g_{\mathbb{S}^{n - 1}}$. The space of $C^{k, \alpha}_0 (\hat{\Sigma}_1)$ is
  the subspace of $C^{k, \alpha} (\hat{\Sigma}_1)$ with zero averages. We
  extend the definition of $F (t, v)$ to $t = 0$ by taking limits and we see
  that $F (0, v) = - \bar{\Delta}_1 v$ from \eqref{limit perturb}. Evidently,
  \[ D F_{(0, 0)} (0, v) = \tfrac{\partial}{\partial s} F (0, s v) |_{s = 0} =
     - \bar{\Delta}_1 v. \]
  This says that $D F_{(0, 0)}$ is an isomorphism restricted to $C^{0,
  \alpha}_0 (\hat{\Sigma}_1)$. By the implicit function theorem, there exists
  a family of functions $v (\cdot, \text{} t)$ for all small $t$ that $F (t, v
  (\cdot, t)) = 0$, which gives that $\hat{\eta}_{t, 1, t v (\cdot, t)}$ is a
  constant for such $v (\cdot, t)$.
\end{proof}

Now we need to determine the sign of $\hat{\eta}$.

\begin{proposition}
  Under the assumption of isometric tangent cones and the scalar curvature
  comparison, $\hat{\eta} (t) := \hat{\eta}_{t, 1, t v (\cdot, t)} \leqslant
  0$ for all sufficiently small $t$ where $v_t$ are the functions constructed
  above.
\end{proposition}

Before proving this proposition, we discuss the previous strategy (cf.
{\cite{chai-scalar-2025}}): we integrate \eqref{taylor reduced} on
$\hat{\Sigma}_1$ with respect to the metric $\mathrm{d} \tau^2 + A^2 \tau^2
g_{\mathbb{S}^{n - 1}}$,
\begin{equation}
  t^{- 1} \hat{\eta} (t) = \int_{\hat{\Sigma}_1} \hat{\eta}_{t, 1, 0} t^{- 1}
  + O (t) . \label{integrate taylor reduced}
\end{equation}
The limit of $\hat{\eta}_{1, 0}$ is zero and note that the term
$\hat{\eta}_{1, 0} t^{- 1}$ is a first variation of \eqref{eta hat t
parameter} at $t = 0$ for the metric $\hat{g}_t$ and the functions
$\hat{u}_t$. As easily seen, we have a comparison of the metric by the
distance non-increasing property, but there is no comparison of the functions
$\hat{u}_t$. It turns out that we can remedy the issue by considering the
weighted version of \eqref{integrate taylor reduced}.

We introduce some notations. Let $g_t$ and $u_t$ be a family of metrics and
functions index by $t \in (- \varepsilon {,} \varepsilon)$ such that $g_0 = g$
and $u_0 = u$. Denote $\delta g = \lim_{t \to 0} t^{- 1} g_t$ and $\delta u =
\lim_{t \to 0} t^{- 1} u_t$ (derivatives at $t = 0$ of $g_t$ and $u_t$). Let
$G (g, u)$ be any geometric quantity which depends on $g$ and $u$ and $G (g_t,
u_t)$ be the corresponding geometric quantity computed with respect to $g_t$
and $u_t$, then $\delta_u G$, the variation of $G$ with only $u$ varying is
defined to be $\lim_{t \to 0} t^{- 1} (G (g, u_t) - G (g, u))$. We can define
similarly $\delta_g G$. Let $(g_t^{(i)}, u_t^{(i)})$, $i = 1, 2$ be any two
families of metrics and functions, we use $\delta^{(i)}$ to denote the
variation is taken with respect to each family index by $i$.

\begin{proposition}
  \label{variational principle}Let $g_t$ and $u_t$ be a family of metrics and
  functions on a manifold $M$ with non-empty boundary $\partial M$, then
  \begin{equation}
    \int_{\partial M} \delta_u (H + \gamma u^{- 1} u_{\nu}) u^2 + \int_M u^2
    \delta_u \Lambda_g = 0, \label{var wrt u}
  \end{equation}
  and
  \begin{equation}
    \int_{\partial M} \delta_g (H + \gamma u^{- 1} u_{\nu}) u^2 + \int_M u^2
    \delta_g \Lambda_g = - \tfrac{1}{2} \int_M u^2 \langle \mathcal{R}_g,
    \delta g \rangle - \tfrac{1}{2} \int_{\partial M} u^2 \langle
    \mathcal{A}_g, \delta g \rangle . \label{var wrt g}
  \end{equation}
  Here, the integration is with respect to the metric $g$ and the tensors
  $\mathcal{R}$ and $\mathcal{A}$ are defined by
\begin{align}
\mathcal{R}_g & =\ensuremath{\operatorname{Ric}}_g - 2 \gamma
\tfrac{\nabla^2 u}{u} + 2 \gamma \frac{\nabla (u \nabla u)}{u^2}
\label{mod ric} \\
& \quad - \gamma \frac{\nabla_k (u \nabla^k u)}{u^2} g - \frac{\nabla^2
u^2}{u^2} + \frac{\Delta_g u^2}{u^2} g, \\
\mathcal{A}_g = & A_{\partial M} - (2 - \gamma) u^{- 1} u_{\nu}
g_{\partial M}
\end{align}
  where $\nabla u$ is understood as a 1-form $(\nabla u) (e_i) = \nabla_{e_i}
  u$.{}
\end{proposition}

\begin{proof}
  Since both $\delta_u H$ and $\delta_u R_g$ vanish, \eqref{var wrt u} is
  equivalent to
  \[ \int_{\partial M} \delta_u (u^{- 1} u_{\nu}) u^2 + \int_M \delta_u (-
     u^{- 1} \Delta_g u) u^2 = 0. \]
  By direct calculation and integration by parts,
\begin{align}
& \int_M \delta_u (- u^{- 1} \Delta_g u) u^2 \\
= & - \int_M (- u^{- 2} \delta u \Delta_g u + u^{- 1} \Delta_g (\delta u))
u^2 \\
= & \int_M (\delta u \Delta_g u - \delta u \Delta_g u^1) \\
& \quad + \int_{\partial M} (- u  (\delta u)_{\nu} + (u)_{\nu} \delta u)
\end{align}
  The interior terms of the above vanish and the boundary term is precisely
  \[ - \int_M u^2 \delta_u (u^{- 1} u_{\nu}) = - \int_M u^2 (- u^{- 2} \delta
     u u_{\nu} + u^{- 1} (\delta u)_{\nu}) . \]
  Hence, \eqref{var wrt u} is proved.
  
  The calculation of \eqref{var wrt g} is more involved but direct. We start
  with the well known variation
  \[ \delta R_g = - \langle \ensuremath{\operatorname{Ric}}, \delta g \rangle
     +\ensuremath{\operatorname{div}}^2_g (\delta g) - \Delta_g
     \ensuremath{\operatorname{tr}}_g (\delta g) \]
  of $R_g$ and the variation
  \[ 2 \delta_g H_g = (\mathrm{d} (\ensuremath{\operatorname{tr}}_g (\delta
     g)) -\ensuremath{\operatorname{div}}_g (\delta g)) (\nu)
     -\ensuremath{\operatorname{div}}_{\partial M} (\delta g (\cdot,
     \nu))^{\top} - \langle A_{\partial M}, \delta g \rangle \]
  of the mean curvature $H_g$. For the variation $\delta_g (u^{- 1} \Delta_g
  u)$, we need the variation of Christoffel symbols $\Gamma_{i j}^k$ of $g$,
  which is given by
  \[ \delta_g \Gamma_{i j}^k = \tfrac{1}{2} g^{k l} (\nabla_i (\delta g)_{j l}
     + \nabla_j (\delta g)_{i l} - \nabla_l (\delta g)_{i j}), \]
  and it yields
\begin{align}
& \delta_g (\Delta_g u) = \delta_g (g^{i j} \nabla_i \nabla_j u)
\\
= & \delta_g (g^{i j} (\partial_i \partial_j u - \Gamma_{i j}^k \partial_k
u)) \\
= & - \delta_g g^{i j} (\partial_i \partial_j u - \Gamma_{i j}^k
\partial_k u) - g^{i j} \delta_g \Gamma_{i j}^k \partial_k u \\
= & - \langle \delta g, \nabla^2 u \rangle - \tfrac{1}{2} \nabla^k u (2
\nabla^i (\delta g)_{i k} - \nabla_k (\ensuremath{\operatorname{tr}}_g
(\delta g))) .
\end{align}
  The variation $\delta_g (| \nabla u|^2) = - \langle \delta g, \nabla u
  \otimes \nabla u \rangle$. Hence,
\begin{align}
& \int_M u^2 \delta_g \Lambda \\
= & \int_M \delta_g (- \gamma u^{- 1} \Delta_g u + \tfrac{1}{2} R_g) u^{2
- 2 \sigma} \\
= & - \gamma \int_M u  (- \langle \delta g, \nabla^2 u \rangle -
\tfrac{1}{2} \nabla^k u (2 \nabla^i (\delta g)_{i k} - \nabla_k
(\ensuremath{\operatorname{tr}}_g (\delta g)))) \\
& \quad + \tfrac{1}{2} \int_M u^2 (- \langle
\ensuremath{\operatorname{Ric}}, \delta g \rangle
+\ensuremath{\operatorname{div}}^2_g (\delta g) - \Delta_g
\ensuremath{\operatorname{tr}}_g (\delta g)) \\
= & \gamma \int_M u \langle \delta g, \nabla^2 u \rangle \\
& \quad + \gamma \int_{\partial M} u (\delta g) (\nabla u, \nu) - \gamma
\int_M (\delta g)_{i k} \nabla^i (u \nabla^k u) \\
& \quad - \tfrac{\gamma}{2} \int_{\partial M} u u_{\nu}
\ensuremath{\operatorname{tr}}_g (\delta g) + \tfrac{\gamma}{2} \int_M
\nabla_k (u \nabla^k u) \ensuremath{\operatorname{tr}}_g (\delta g)
\\
& \quad - \tfrac{1}{2} \int_M u^2 \langle
\ensuremath{\operatorname{Ric}}, \delta g \rangle \\
& \quad + \tfrac{1}{2} \int_{\partial M} (u^2
\ensuremath{\operatorname{div}}_g (\delta g) (\nu) - (\delta g) (\nabla
u^2, \nu)) + \tfrac{1}{2} \int_M \langle \nabla^2 u^2, \delta g \rangle
\\
& \quad + \tfrac{1}{2} \int_{\partial M} (\nabla_{\nu} u^2
\ensuremath{\operatorname{tr}}_g (\delta g) - u^2 \nabla_{\nu}
\ensuremath{\operatorname{tr}}_g (\delta g)) - \tfrac{1}{2} \int_M
\ensuremath{\operatorname{tr}}_g (\delta g) \Delta_g u^2 .
\end{align}
  In the last equality, we have used integration by parts so that the interior
  terms do not contain the derivatives of $\delta g$.
  
  Now we turn to $\delta_g (H + \gamma u^{- 1} u_{\nu})$. Recall that
  \[ 2 \delta_g H = (\mathrm{d} (\ensuremath{\operatorname{tr}}_g (\delta g))
     -\ensuremath{\operatorname{div}}_g (\delta g)) (\nu)
     -\ensuremath{\operatorname{div}}_{\partial M} (\delta g (\nu,
     \cdot))^{\top} - \langle \delta g, A_{\partial M} \rangle, \]
  and $\delta_g (\nu^j) = - ((\delta g (\nu, \cdot))^{\top})^j - \tfrac{1}{2}
  \delta g (\nu, \nu) \nu^j$ (see {\cite{miao-mass-2022}}). Here, $\delta g
  (\nu, \cdot)^{\top}$ is understood as the dual vector field of the
  tangential component of the 1-form $\delta g (\nu, \cdot)$. So
\begin{align}
& \int_{\partial M} u^2 \delta_g (H + \gamma u^{- 1} u_{\nu}) \\
= & \tfrac{1}{2} \int_{\partial M} u^2 ((\mathrm{d}
\ensuremath{\operatorname{tr}}_g (\delta g)
-\ensuremath{\operatorname{div}}_g (\delta g)) (\nu)
-\ensuremath{\operatorname{div}}_{\partial M} (\delta g (\nu,
\cdot))^{\top} - \langle \delta g, A_{\partial M} \rangle) \\
& \quad - \gamma \int_{\partial M} u \delta g (\nu, \nabla^{\partial M}
u) - \tfrac{\gamma}{2} \int_{\partial M} u \delta g (\nu, \nu) u_{\nu}
\\
= & \tfrac{1}{2} \int_{\partial M} u^2 ((\mathrm{d}
\ensuremath{\operatorname{tr}}_g (\delta g)
-\ensuremath{\operatorname{div}}_g (\delta g)) (\nu) - \langle \delta g,
A_{\partial M} \rangle) + \tfrac{1}{2} \int_{\partial M} \delta g (\nu,
\nabla^{\partial M} u^2) \\
& \quad - \gamma \int_{\partial M} u \delta g (\nu, \nabla^{\partial M}
u) - \tfrac{\gamma}{2} \int_{\partial M} u \delta g (\nu, \nu) u_{\nu} .
\end{align}
  In the last equality, we have used integration by parts on $\partial M$. By
  simply collecting $\int_M u^2 \delta_g \Lambda$ and $\int_{\partial M}
  \delta_g (H + \gamma u^{- 1} u_{\nu})$,
  \[ \int_{\partial M} \delta_g (H + \gamma u^{- 1} u_{\nu}) u^2 + \int_M u^2
     \delta_g \Lambda_g = \tfrac{1}{2} \int_M u^2 \mathcal{R}_g + \tfrac{1}{2}
     \int_{\partial M} u^2 \mathcal{A}_1 \]
  with $\mathcal{A}_1$ given by
\begin{align}
\mathcal{A}_1 & = A_{\partial M} - 2 \gamma u^{- 1} \nabla u \otimes \nu +
\gamma u^{- 1} u_{\nu} g + \tfrac{\nabla u^2}{u^2} \otimes \nu \\
& \quad - \frac{\nabla_{\nu} u^2}{u^2} g - \frac{\nabla^{\partial M}
u^2}{u^2} \otimes \nu + 2 \gamma u^{- 1} \nu \otimes \nabla^{\partial M} u
+ \gamma u^{- 1} \nu \otimes \nu .
\end{align}
  Checking that $\mathcal{A}_1 =\mathcal{A}_g$ finishes the proof the
  proposition.
\end{proof}

By taking the difference of two pairs $(g_t^{(1)}, u_t^{(1)})$ and
$(g_t^{(2)}, u_t^{(2)})$ which tend to the same limit, we easily obtain the
following.

\begin{corollary}
  \label{infinitesimal comparison}Let $(M, g)$ be a smooth manifold with
  non-empty boundary and $u$ be a smooth function on $M$ such that
  $\mathcal{R}_g \geqslant 0$ and $\mathcal{A}_g \geqslant 0$. Assume that the
  pair $(g_t^{(1)}, u_t^{(1)})$ and $(g_t^{(2)}, u_t^{(2)})$ which tends to
  the same limit $(g, u)$ as $t \to 0$, and $g_t^{(1)} \geqslant g_t^{(2)}$
  for all $t$ small, then
  \begin{equation}
    \int_{\partial M} (\delta_u^{(2)} (H + \gamma u^{- 1} u_{\nu}) -
    \delta_u^{(1)} (H + \gamma u^{- 1} u_{\nu})) u^2 + \int_M u^2
    (\delta_u^{(2)} \Lambda_g - \delta_u^{(1)} \Lambda_g) = 0,
  \end{equation}
  and
  \begin{equation}
    \int_{\partial M} (\delta_g^{(2)} (H + \gamma u^{- 1} u_{\nu}) -
    \delta_g^{(1)} (H + \gamma u^{- 1} u_{\nu})) u^2 + \int_M u^2
    (\delta_g^{(2)} \Lambda_g - \delta_g^{(1)} \Lambda_g) \geqslant 0.
  \end{equation}
\end{corollary}

Now we compute the $\mathcal{R}_{\bar{g}}$ and $\mathcal{A}_{\bar{g}}$ for the
tangent cone metric $\bar{g} = \mathrm{d} t^2 + A^2 t^2 g_{\mathbb{S}^{n -
1}}$ and the function $\bar{u}$.

\begin{lemma}
  Let $0 < \gamma < 2$, $\bar{g} = \mathrm{d} t^2 + A^2 t^2 g_{\mathbb{S}^{n -
  1}}$ and $\bar{u} (t) = t^{\tfrac{1}{2 - \gamma}}$, then
  $\mathcal{R}_{\bar{g}} \geqslant 0$ and $\mathcal{A}_{\bar{g}} \geqslant 0$.
\end{lemma}

\begin{proof}
  Let $e_i$ be a vector orthogonal to $\partial_t$ which has length with
  respect to the metric $\bar{g}$ and $\bar{\nabla}$ be the connection with
  respect to the metric $\bar{g}$. The Ricci tensor of $\bar{g}$ is given by
  $\ensuremath{\operatorname{Ric}}_{\bar{g}} (\partial_t, \partial_t) = 0$ and
  $\ensuremath{\operatorname{Ric}}_{\bar{g}} (e_i, e_i) = (n - 2) (A^{- 2} -
  1) t^{- 2}$. Let $\alpha = \tfrac{1}{2 - \gamma}$, then $\bar{u} (t) =
  t^{\alpha}$. Then $\bar{\nabla}^2 \bar{u} (\partial_t, \partial_t) = \alpha
  (\alpha - 1) t^{\alpha - 2}$, $\bar{\nabla}^2 u (e_i, e_i) = \alpha
  t^{\alpha - 2}$ and $\Delta_{\bar{g}} u = \alpha (\alpha + n - 2) t^{\alpha
  - 2}$. Using these in the expression of $\mathcal{R}_{\bar{g}}$ and then
  $\alpha = \tfrac{1}{2 - \gamma}$, we find that
\begin{align}
t^2 \mathcal{R}_{\bar{g}} (\partial_t, \partial_t) & = \frac{2 (n - 1) -
(n - 2) \gamma}{2 - \gamma}, \\
t^2 \mathcal{R}_{\bar{g}} (e_i, e_i) & = (n - 2) (A^{- 2} - 1) + \tfrac{2
(n - 1)}{2 - \gamma} .
\end{align}
  By the assumption of Theorem \ref{llarull cone theorem}, see \eqref{unit
  cross-section size}, $\mathcal{R}_{\bar{g}} \geqslant 0$. And it is not
  difficult to see that $\mathcal{A}_{\bar{g}} = 0$.
\end{proof}

\begin{remark}
  The authors have used a computer algebra system to assist the computation.
\end{remark}

With the help of this corollary, we give the proof of Theorem \ref{llarull
cone theorem}.

\begin{proof}[Proof of Theorem \ref{llarull cone theorem}]
  Now let $(g_t^{(1)}, u_t^{(1)}) = (\hat{g}_t, \hat{u}_t)$ and $(g^{(2)}_t,
  u_t^{(2)}) = (\hat{g}_{\xi, t}, \hat{u}_{\xi, t})$. We multiply
  \eqref{taylor reduced} by $\bar{u}^2$, and by an application of the
  divergence theorem,
  \[ \hat{\eta} (t) t^{- 1} \bar{u}^2 | \hat{\Sigma}_1 | =
     \int_{\hat{\Sigma}_1} \bar{u}^2 \hat{\eta}_{t, 1, 0} t^{- 1} + O (t), \]
  where we have used that $\bar{u}$ is constant on $\hat{\Sigma}_1$. Now $t^{-
  1} \hat{\eta}_{t, 1, 0}$ can be interpreted as
  \[ \delta_g^{(2)} (H + \gamma u^{- 1} u_{\nu}) - \delta_g^{(1)} (H + \gamma
     u^{- 1} u_{\nu}) + (\delta_u^{(2)} (H + \gamma u^{- 1} u_{\nu}) -
     \delta_u^{(1)} (H + \gamma u^{- 1} u_{\nu})) \]
  along $\hat{\Sigma}_1$. By Corollary \ref{infinitesimal comparison} and the
  assumptions on comparisons $g \geqslant g_{\xi}$ and $\Lambda_{g, u}
  \geqslant \Lambda (g_{\xi}, u_{\xi})$, we have that $\hat{\eta} (t)
  \leqslant O (t)$, which implies that $\lim_{t \to 0} \hat{\eta} (t)
  \leqslant 0$.
  
  By the argument of Proposition \ref{sign of leaf}, $\hat{\eta} (t)
  \leqslant 0$ for all sufficiently small $t > 0$. By rescaling back, we
  obtain that each leaf of the foliation $\{\Sigma_t \}_{t \in (0 {,}
  \varepsilon)}$ in Proposition \ref{foliation naer conical point} satisfies
  $\eta (t) \leqslant 0$. Each choice of $\Sigma_t$ gives a lower barrier. If
  there are two conical points, we may apply this construction twice and
  obtain both the upper barrier and the lower barrier. In either case, Theorem
  \ref{llarull cone theorem} are reduced to Theorem \ref{spec llarull band},
  see Remark \ref{slightly general}. By choosing $t$ smaller, we obtain the
  global rigidity.
\end{proof}

\subsection{Llarull's theorem}

Now we prove Theorem \ref{cone llarull}.

\begin{proof}[Proof of Theorem \ref{cone llarull}]
  We briefly sketch an argument of Chai-Wang (cf. {\cite[Section
  3.3]{chai-scalar-2025}}) which reduces the theorem to Theorem \ref{spec
  llarull band} with $\gamma = 0$.
  
  Let $N$ be the north pole of $S_I^n$, and $U$ be the neighborhood of $N$
  such that the $I$-direction coordinate is less than $t$. We consider $f^{-
  1} (U\backslash\{N\})$. The metric of $g$ in $f^{- 1} (U\backslash\{N\})$
  takes the form of $\mathrm{d} r^2 + r^2 g_S$ where $S$ is some manifold of
  dimension $n - 1$ and $g_S$ be the induced metric, and by the metric
  comparison, $r \geqslant t + o (t)$. (Note that $t = O (r)$.) We can see
  that the mean curvature of the $r$-level set $\Sigma_r$ is $\tfrac{n - 1}{r}
  + o (r^{- 1})$, by the relation $r \geqslant t + o (r)$, $\Sigma_r$ is an
  approximate (non-strict) barrier in the sense that $\tfrac{n - 1}{r}
  \leqslant \tfrac{n - 1}{t} + o (r^{- 1})$. In fact, $\Sigma_r$ can be
  perturbed into a strict barrier if $\lim_{r \to 0^+} \tfrac{r}{t} \gneqq 1$.
  Hence, $\lim_{r \to 0} \tfrac{r}{t} = 1$. In that case, taking the tangent
  cones $C_1$ and $C_2$ at $N$ and $p$, we see there exists a map from the
  unit cross-section of $C_1$ to the unit cross-section of $C_2$, which is of
  non-zero degree and of distance non-increasing. Since the scalar curvature
  comparison is preserved by taking the tangent cones and the comparison
  descents to the unit cross-section. Hence the lower dimensional Llarull's
  Theorem \ref{std llarull} shows that the cross-sections are isometric.
  Hence, we showed that $(M, g)$ and $(\mathbb{S}_I^n, g_{\xi})$ have
  isometric tangent cones at $p \in f^{- 1} (N)$ and $N$. Using Theorem
  \ref{spec llarull band} with $u$ being constant finishes the proof.
\end{proof}

\bibliographystyle{alpha}
\bibliography{4dllarull}

\end{document}